\documentclass[a4paper,twoside,10pt]{article}

\usepackage[a4paper,left=3cm,right=3cm, top=3cm, bottom=3cm]{geometry}
\usepackage[latin1]{inputenc}
\usepackage{mathrsfs}
\usepackage{mathtools}
\usepackage{graphicx}
\usepackage{epstopdf}
\graphicspath{{figs/}}
\usepackage{amsmath}
\usepackage{amsthm}
\usepackage{amssymb}
\usepackage{stmaryrd}
\usepackage{float}
\usepackage{bigints}
\usepackage{cite}
\usepackage{color}
\usepackage[abs]{overpic}
\usepackage[font=footnotesize,labelfont=bf]{caption}
\usepackage{cases}
\usepackage{tikz}
\usepackage{algorithmic}
\usepackage{rotating}
\usepackage{blkarray}
\usetikzlibrary{matrix,calc,arrows}
\usepackage[author={Lorenzo}]{pdfcomment}
\usepackage{soul,xcolor}
\usepackage{enumitem}  
\usepackage{verbatim}
\usepackage{graphicx}
\usepackage{subcaption}
\usepackage{mwe}
\usepackage{algorithm}

\restylefloat{table}
\theoremstyle{plain}
\newtheorem{thm}{Theorem}[section]
\newtheorem{cor}[thm]{Corollary}
\newtheorem{lem}[thm]{Lemma}

\theoremstyle{definition}

\theoremstyle{remark}

\newcommand{\Norm}[1]{{\left\|{#1} \right\|}}
\newcommand{\SemiNorm}[1]{{\left|{#1} \right|}}

\newcommand{\qp}{q_\p}
\newcommand{\E}{K}
\newcommand{\p}{p}
\newcommand{\h}{h}

\newcommand{\e}{e}
\newcommand{\hE}{\h_\E}

\newcommand{\he}{\h_\e}
\newcommand{\taun}{\mathcal T_\h}

\newcommand{\uh}{u_\h}
\newcommand{\vh}{v_\h}
\newcommand{\Vh}{V_\h}
\newcommand{\VhE}{\Vh(\E)}
\newcommand{\qpmt}{q_{\p-2}}
\newcommand{\qpmoe}{q_{\p-1}^\e}
\newcommand{\qpmtE}{\qpmt^\E}
\newcommand{\qpE}{\qp^\E}
\newcommand{\Pbb}{\mathbb P}
\newcommand{\Rbb}{\mathbb R}

\newcommand{\SE}{S^\E}
\newcommand{\Piz}{\Pi^{0}_{\p-2}}
\newcommand{\Pizp}{\Pi^{0}_{\p}}
\newcommand{\Pize}{\Pi^{0,\e}_{\p-1}}
\newcommand{\Pinabla}{\Pi^{\nabla}_{\p}}
\newcommand{\malphae}{m_\alpha^\e}
\newcommand{\boldalpha}{\boldsymbol \alpha}
\newcommand{\mboldalpha}{m_{\boldalpha}}
\newcommand{\xbf}{\mathbf x}
\newcommand{\xE}{\xbf_\E}

\newcommand{\aE}{a^\E}
\newcommand{\ahE}{\aE_\h}
\newcommand{\NE}{N_\E}
\DeclareMathOperator{\dof}{dof}
\newcommand{\nbf}{\mathbf n}
\newcommand{\nbfE}{\nbf_\E}
\newcommand{\EcalE}{\mathcal E^\E}
\newcommand{\etaE}{\eta_\E}
\DeclareMathOperator{\HOT}{HOT}
\newcommand{\omegaE}{\omega_\E}
\newcommand{\vI}{v_I}
\newcommand{\tauE}{\mathcal T(\E)}
\newcommand{\vC}{v_C}
\newcommand{\etatilde}{\widetilde \eta}
\newcommand{\etatildeE}{\etatilde_\E}

\date{}
\title{\normalsize{The role of stabilization in the virtual element method: a survey}}
\author{\small{Lorenzo Mascotto \!\!\!\thanks{
Department of Mathematics and Applications, University of Milano-Bicocca, 20125 Milan, Italy
(lorenzo.mascotto@unimib.it);
IMATI-CNR, Pavia, Italy;
Faculty of Mathematics, University of Vienna, 1090 Vienna, Austria (lorenzo.mascotto@univie.ac.at)}}}
\date{}

\begin{document}
\maketitle

\begin{abstract}
\noindent  The virtual element method was introduced 10 years ago
and it has generated a large number of theoretical results and applications ever since.
Here, we overview the main mathematical results
concerning the stabilization term of the method
as an introduction for newcomers in the field.
In particular, we summarize the proofs of some results for two dimensional ``nodal''
conforming and nonconforming virtual element spaces
to pinpoint the essential tools used in the stability analysis.
We discuss their extensions to several other virtual elements.
Finally, we show several ways to prove interpolation estimates,
including a recent one that is based on
employing the stability bounds.
 
\medskip\noindent
\textbf{AMS subject classification}: 65N12; 65N15.

\medskip\noindent
\textbf{Keywords}: virtual element method;
stability bound; interpolation estimate.
\end{abstract}

\section{Introduction} \label{section:introduction}

The virtual element method was introduced ten years ago~\cite{BeiraodaVeiga-Brezzi-Cangiani-Manzini-Marini-Russo:2013}
as a generalization of the finite element method to polytopic meshes.
Local virtual element spaces consist of solutions to local problems with polynomial data.
Therefore, virtual element functions are not known in closed form;
in the spirit of the mimetic finite differences~\cite{Lipnikov:Manzini-Shashkov:2014, BeiraodaVeiga-Lipnikov-Manzini:2014},
only the evaluation of their degrees of freedom is required in the design of the method.

Consequently, the bilinear forms appearing in the variational formulation
of given partial differential equations are not computable
and have to be approximated based on two main ingredients:
projections from local virtual element spaces onto polynomial spaces;
bilinear forms that stabilize the scheme.

This entails that the error analysis for virtual elements
has the form of a Strang-type result,
where several variational crimes have to be taken into account.
In particular, one has to cope with certain
stability bounds and interpolation estimates in virtual element spaces.
\medskip

The first work on the virtual element method
contains the following statement concerning the stability estimates
\cite[Section~4.6]{BeiraodaVeiga-Brezzi-Cangiani-Manzini-Marini-Russo:2013}:
\medskip

\fbox{\begin{minipage}{\textwidth}
\noindent ``\emph{In general, the choice of the bilinear form~$\SE(\cdot,\cdot)$
[the local virtual element stabilization]
would depend on the problem and on the degrees of freedom.
From $(4.20)$ it is clear that~$\SE(\cdot,\cdot)$
must scale like~$\aE(\cdot, \cdot)$
[the ``grad-grad bilinear form]
on the kernel of $\Pinabla$ [an $H^1$ polynomial projector].
Choosing then the canonical basis $\varphi_1, \dots, \varphi_{\NE}$
as
\[
\chi_i(\varphi_j) = \delta_{i,j}, \quad
i,j =1, \dots, \NE,
\qquad \text{[$\chi_i$ is the $i$-th local degree of freedom]}
\]
the local stiffness matrix is given by
\[
\ahE(\varphi_i,\varphi_j)
= \aE(\Pinabla \varphi_i, \Pinabla \varphi_j)
  + \SE( (I-\Pinabla) \varphi_i , (I-\Pinabla) \varphi_j).
\]
In our case it is easy to check that,
on a ``reasonable'' polygon,
$\aE(\varphi_i,\varphi_j) \approx 1$.
Note that this holds true for all
$i = 1, 2, \dots,\NE$ [$\NE$ is the dimension of the local discrete space]
since we defined the local degrees of freedom suitably. [\dots]
However, several types of misbehaviour can occur for awkwardly-shaped polygons,
in particular if two or more vertices tend to coalesce, although, in our numerical experiments,
the method appears to be quite robust in this respect.}''
\end{minipage}}
\medskip

So, some facts were made clear since the very inception of the method:
\begin{itemize}
    \item the stabilization is not required to have approximation properties:
    it only has to scale as the corresponding ``continuous'' bilinear form;
    \item a reasonable choice for the stabilization is the ``dofi-dofi'' one given by
    $\SE(\varphi_i,\varphi_j)=\delta_{i,j}$
    for all~$i$, $j=1,\dots,\NE$;
    \item the ``dofi-dofi'' stabilization needs to be carefully tuned/changed in presence of ``\emph{awkwardly-shaped polygons}''.
\end{itemize}
Similar considerations are contained in other works
tracing back to the early years of the virtual elements literature,
say, the period 2013-2016:
there, one can find heuristic motivations but no proofs
of the stability bounds, i.e., bounds of the form
$\alpha^* \aE(\vh,\vh) \le \SE(\vh,\vh) \le \alpha^* \aE(\vh,\vh)$
for suitable discrete functions~$\vh$
and positive constants~$\alpha_* < \alpha^*$.

In~2017, the first paper~\cite{BeiraodaVeiga-Lovadina-Russo-2017}
on the theoretical aspects of the stabilization in
virtual elements was published:
stability properties for two dimensional nodal conforming virtual elements
were investigated on rather general geometries.
Ever since, several related contributions have been proposed;
amongst them we mention the other three pioneering works~\cite{Chen-Huang:2018, Brenner-Guan-Sung:2017, Brenner-Sung:2018}.
Most of the literature on this topic
is concerned with nodal conforming virtual elements;
fewer works can be found on other types of virtual elements
such as face, edge, Stokes, immersed-like virtual elements.
\medskip

In light of this,
the virtual element literature on the theoretical aspects of the stabilization
might be classified into three periods:
\begin{itemize}
    \item the early years (2013-2016), when the properties of the stabilization were introduced and motivated heuristically;
    \item the pioneering years (2017-2018),
    when the first papers~\cite{BeiraodaVeiga-Lovadina-Russo-2017, Brenner-Guan-Sung:2017, Brenner-Sung:2018, Chen-Huang:2018}
    on the theoretical aspects of the stabilization were published;
    \item the consolidation years (2019-2023),
    when several other works on the topic were written,
    the pioneering analysis was generalized,
    and other types of virtual elements were considered.
\end{itemize}

We deemed useful to collect and review
all contributions that we are aware of
about the theoretical aspects of the stabilization.
We hope that this work could represent a starting guide in this area.
More details and the topics we are going to cover
can be found at the end of this section.
\medskip

\noindent\textbf{Disclaimer.}
As we are interested here only in reviewing the literature on theoretical aspects of the virtual element stabilization,
we do not review contributions on more ``practical'' issues
as those in~\cite{Wriggers-Hudobivnik:2017, Wriggers-Reddy-Rust-Hudobivnik:2017, BeiraodaVeiga-Dassi-Russo:2017, Mascotto:2018, Dassi-Mascotto:2018, Russo-Sukumar:2023}.

\paragraph*{Notation and assumptions.}
Given~$D \subset \Rbb^d$, $d=1,2,3$, a Lipschitz domain
with measure~$\vert D \vert$,
we introduce the Sobolev space~$H^s(D)$, $s\ge 0$,
and endow it with the usual norm~$\Norm{\cdot}_{s,D}$,
seminorm~$\SemiNorm{\cdot}_{s,D}$,
and bilinear form~$(\cdot,\cdot)_{s,D}$.
If~$s=0$, we let the Sobolev space~$H^0(D)$ be the Lebesgue space~$L^2(D)$.
Negative order Sobolev spaces are defined by duality.

We set the spaces~$\Pbb_\p(\E)$ and~$\Pbb_\p(\e)$
of polynomials of maximum degree~$\p$
over a polygon~$\E$ and an interval~$\e$.
We use the notation~$\Pbb_{-1}(\E)= \{ 0 \}$.
For the space~$\Pbb_\p(\E)$,
given~$\xE = (x_\E,y_\E)$ and~$\hE$ denote the centroid and diameter of~$\E$,
it is convenient to introduce the basis~$\{ \mboldalpha \}$ of scaled and centred monomials
\begin{equation} \label{monomials}
\mboldalpha(\xbf)
\!:=\! \left( \frac{\xbf-\xE}{\hE}   \right)^{\boldalpha}
\!\!=\! \left( \frac{x - x_\E}{\hE} \right)^{\alpha_1}
  \! \left( \frac{y - y_\E}{\hE} \right)^{\alpha_2},
\; \boldalpha = (\alpha_1,\alpha_2) \in \Rbb^2,\;
\vert \boldalpha \vert=\alpha_1+\alpha_2
= 0,\dots,\p.
\end{equation}
Analogously, we can define the basis~$\{ \malphae \}$, $\alpha=1,\dots,\p$,
of shifted and scaled monomials for~$\Pbb_{\p}(\e)$.

Given two quantities~$a$ and~$b$,
we write~$a \lesssim b$ meaning that
there exists~$c$ depending on the shape of~$\E$,
but not on its size, such that~$a \le c \ b$.

\paragraph*{Outline.}
We prove the stability bounds
for two dimensional nodal conforming virtual elements in Section~\ref{section:nodal}:
here, we focus on two paradigmatic stabilizations;
derive corresponding stability bounds;
underline what the main tools in the proof are;
review the relevant literature.
We discuss the generalization to nonconforming ``nodal'' elements
and arbitrarily regular ``elliptic-type'' elements in Section~\ref{section:ncVEM}.
``Nonelliptic-type'' elements,
such as face, edge, Stokes-like, and immersed virtual elements,
are overviewed in Section~\ref{section:non-elliptic}.
Comments on ``$\p$-version'' spaces,
and the role of the stabilization in relability and efficiency bounds
for residual error estimators
are given in Section~\ref{section:p-version-adaptivity}.
In Section~\ref{section:interpolation},
we show that the stability bound easily implies interpolation estimates.
We draw some conclusions in Section~\ref{section:conclusions}.

\section{Basic results: stability in nodal conforming virtual elements} \label{section:nodal}
In this section, we review the role of the stabilization in two dimensional nodal conforming virtual elements.
After designing local virtual elements with a set of unisolvent degrees of freedom,
we describe computable polynomial projectors and discrete bilinear forms in Section~\ref{subsection:nodal-VE}.
We introduce two paradigmatic stabilizations
and provide an easy proof of stability bounds
in Section~\ref{subsection:stab-nodal}.
In Section~\ref{subsection:literature-nodal},
we review the essential literature concerning the theory
behind the stabilization in nodal conforming virtual elements.

\subsection{Nodal conforming virtual elements} \label{subsection:nodal-VE}
Following~\cite{BeiraodaVeiga-Brezzi-Cangiani-Manzini-Marini-Russo:2013},
given a polygonal element~$\E$ and a positive integer number~$\p$,
we define the nodal conforming virtual element
\begin{equation} \label{c-virtual-element-space}
\VhE:=
\left\{ \vh \in H^1(\E) \ \middle| \
\Delta \vh \in \Pbb_{\p-2}(\E),\;
\vh{}_{|\e} \in \Pbb_\p(\e) \quad
\forall \e \in \EcalE
\right\}.
\end{equation}
We endow the space~$\VhE$ with the following set of unisolvent degrees of freedom:
given~$\vh$ in~$\VhE$,
\begin{itemize}
    \item the point values of~$\vh$ at the vertices of~$\E$;
    \item on each edge~$\e$ of~$\E$,
    the point values of~$\vh$ at the~$\p-1$ internal Gau\ss-Lobatto nodes;
    \item given the scaled monomial basis~$\{ \mboldalpha \}$
    of~$\Pbb_{\p-2}(\E)$ as in~\eqref{monomials}, the scaled moments
    \[
    \frac{1}{\vert \E \vert} \int_\E \mboldalpha \ \vh.
    \]
\end{itemize}
The second set of degrees of freedom can be replaced
by suitably scaled edge moments
or uniformly spaced nodes.
We collect the above degrees of freedom in the set
$\{ \dof_j \}_{j=1}^{\NE}$,
$\NE$ being the dimension of~$\VhE$.

The degrees of freedom of~$\VhE$ allow for the computation of several polynomial projections~\cite{BeiraodaVeiga-Brezzi-Cangiani-Manzini-Marini-Russo:2013}.
In what follows, we need the operator~$\Pinabla : H^1(\E) \to \Pbb_{\p}(\E)$ defined as
\begin{equation} \label{H1-projector}
\begin{cases}
    (\nabla \qp, \nabla (v-\Pinabla v))_{0,\E}
                \qquad \forall \qp \in \Pbb_\p(\E)\\
    \int_{\partial \E} (v-\Pinabla v) = 0 ,
\end{cases}
\end{equation}
and the operator~$\Piz : L^2(\E) \to \Pbb_{\p-2}(\E)$ defined as
\begin{equation} \label{L2-projector}
(\qpmt, v-\Piz v)_{0,\E} 
\qquad\qquad \forall \qpmt \in \Pbb_{\p-2}(\E).
\end{equation}
Given the degrees of freedom of a function~$\vh$ in~$\VhE$,
we can compute~$\Pinabla \vh$ and~$\Piz \vh$;
see~\cite{BeiraodaVeiga-Brezzi-Cangiani-Manzini-Marini-Russo:2013}.

Other options to fix the constant part of~$\Pinabla$,
see the second condition in~\eqref{H1-projector},
can be found in the literature.
For instance, it is alternatively possible to fix the average in the bulk
or the arithmetic average of the values at the vertices of the polygon.

The standard discretization of the bilinear form
$\aE(\cdot,\cdot) := (\nabla \cdot,\nabla \cdot)_{0,\E}$
in nodal conforming virtual elements is given by
\begin{equation} \label{local-discrete-bf}
\ahE(\uh,\vh) 
:= \aE(\Pinabla \uh, \Pinabla \vh) 
   + \SE ( (I-\Pinabla) \uh, (I-\Pinabla) \vh ).
\end{equation}
The bilinear form~$\SE(\cdot , \cdot) : \VhE \times \VhE$ is required to be computable via the degrees of freedom
and satisfies the following stability bounds:
there exist $0 <\alpha_* \le \alpha^*$
independent of~$\hE$ such that
\begin{equation} \label{stability-bounds}
\alpha_* \SemiNorm{\vh}_{1,\E}^2
\le \SE(\vh, \vh)
\le \alpha_* \SemiNorm{\vh}_{1,\E}^2
\qquad\qquad \forall \vh \in \VhE \cap \ker(\Pinabla).
\end{equation}

\subsection{Two stabilizations for nodal conforming virtual elements} \label{subsection:stab-nodal}
We introduce two stabilizations
that are common in the literature of nodal conforming virtual elements.
The first one is known as the ``dofi-dofi'' stabilization~\cite{BeiraodaVeiga-Brezzi-Cangiani-Manzini-Marini-Russo:2013}
and is given by
\begin{equation} \label{dofi-dofi-stab}
    \SE(\uh,\vh)
    := \sum_{j=1}^{\NE} \dof_j(\uh) \dof_j(\vh) 
    \qquad\qquad \forall \uh, \vh \in \VhE.
\end{equation}
The second one,
which we shall refer to as ``projected'' stabilization, 
is~\cite{BeiraodaVeiga-Chernov-Mascotto-Russo:2018}
\begin{equation} \label{projected-stabilization}
    \SE(\uh,\vh)
    := \hE^{-1} (\uh,\vh)_{0, \partial\E}
       + \hE^{-2} (\Piz \uh, \Piz \vh)_{0,\E}
    \qquad\qquad \forall \uh, \vh \in \VhE.
\end{equation}
Both stabilizations are computable using the degrees of freedom of~$\VhE$.

\begin{lem} \label{lemma:stability}
The bilinear forms~$\SE(\cdot ,\cdot)$ in~\eqref{dofi-dofi-stab} 
and~\eqref{projected-stabilization}
satisfy~\eqref{stability-bounds}.
\end{lem}
\begin{proof}
We prove the assertion for the ``projected'' stabilization~\eqref{projected-stabilization} only;
the details for the ``dofi-dofi'' stabilization~\eqref{dofi-dofi-stab} are similar but slightly more involved.
The assertion for the stabilization in~\eqref{dofi-dofi-stab}
follows combining the bounds for the stabilization in~\eqref{projected-stabilization}
and the techniques, e.g., in~\cite{BeiraodaVeiga-Lovadina-Russo-2017, Brenner-Sung:2018, Chen-Huang:2018}.
\medskip

\noindent\textbf{The lower bound.}
We recall an inverse estimate for functions with polynomial Laplacian;
see, e.g., \cite[Lemma~10]{Cangiani-Georgoulis-Pryer-Sutton:2017}
or~\cite[Theorem~2]{BeiraodaVeiga-Chernov-Mascotto-Russo:2018}:
\begin{equation} \label{inverse-estimate-polynomial-Laplacian}
\Norm{\Delta \vh}_{0,\E} \lesssim \hE^{-1} \SemiNorm{\vh}_{1,\E}
\qquad\qquad
\forall \vh \in H^1(\E) , \; \Delta \vh \in \Pbb_{\p-2}(\E).
\end{equation}
Furthermore,
we have the polynomial inverse estimate
\begin{equation} \label{inverse-estimate-polynomial-boundary}
\SemiNorm{\qp}_{\frac12,\partial\E}
\lesssim \hE^{-\frac12}
            \Norm{\qp}_{0,\partial\E}
\qquad\qquad \forall \qp \in \mathcal C^0(\partial\E),\quad
\qp{}_{|\e} \in \Pbb_{\p}(\e),
\quad \forall \e \in \EcalE.
\end{equation}
Integrating by parts,
recalling that~$\Delta\vh$ belongs to~$\Pbb_{\p-2}(\E)$,
applying the Cauchy-Schwarz inequality,
invoking the definition of negative Sobolev norms,
and using the inverse estimates~\eqref{inverse-estimate-polynomial-Laplacian}
and~\eqref{inverse-estimate-polynomial-boundary},
the lower bound in~\eqref{stability-bounds} follows:
\[
\begin{split}
\SemiNorm{\vh}^2_{1,\E}
& = - \int_\E \Delta\vh \ \vh 
        + \int_{\partial \E} (\nbfE \cdot \nabla \vh) \vh
= - \int_\E \Delta\vh \ \Piz \vh
        + \int_{\partial \E} (\nbfE \cdot \nabla \vh) \vh \\
& \le \Norm{\Delta\vh}_{0,\E} \Norm{\Piz\vh}_{0,\E}
      + \Norm{\nbfE \cdot \nabla \vh}_{-\frac12,\partial \E} \SemiNorm{\vh}_{\frac12,\partial \E} \\
& \lesssim \Norm{\Delta\vh}_{0,\E} \Norm{\Piz\vh}_{0,\E}
      + (\hE \Norm{\Delta\vh}_{0,\E} + \SemiNorm{\vh}_{1,\E})
      \SemiNorm{\vh}_{\frac12,\partial \E} \\
& \lesssim \SemiNorm{\vh}_{1,\E}
    \left( \hE^{-1} \Norm{\Piz \vh}_{0,\E}
      + \hE^{-\frac12} \Norm{\vh}_{0,\partial \E} \right).
\end{split}
\]
\medskip

\noindent\textbf{The upper bound.}
Using the stability of~$\Piz$ in the~$L^2$ norm,
the trace inequality,
and the Poincar\'e inequality
(recall we are assuming~$\vh$
belongs to~$\ker(\Pinabla)$),
the upper bound in~\eqref{stability-bounds} follows:
\[
\SE(\vh,\vh)
:= \hE^{-1} \Norm{\vh}^2_{0,\partial\E}
    + \hE^{-2} \Norm{\Piz \vh}^2_{0,\E}
\le \hE^{-1} \Norm{\vh}^2_{0,\partial\E}
    + \hE^{-2} \Norm{\vh}^2_{0,\E}
\lesssim \SemiNorm{\vh}^2_{1,\E}.
\]

\end{proof}

From the proof of Lemma~\ref{lemma:stability},
we can see that
\begin{itemize}
    \item the lower bound in~\eqref{stability-bounds} follows from (polynomial and virtual) inverse estimates,
    but no zero average condition (based on the fact that $\vh$ belongs also to~$\ker(\Pinabla)$) is used;
    \item the upper bound in~\eqref{stability-bounds}
    is proven based only on ``direct estimates'',
    such as the Poincar\'e inequality and the trace inequality,
    and therefore is valid for $H^1$ functions with zero average.
\end{itemize}

For these reasons, Lemma~\ref{lemma:stability} immediately generalizes as follows.

\begin{cor} \label{corollary:stability}
The bilinear forms~$\SE(\cdot ,\cdot)$ in~\eqref{dofi-dofi-stab} and~\eqref{projected-stabilization}
satisfy the following stability bounds:
there exist~$0< \alpha_* \le \alpha^*$ such that
\[
\alpha_* \SemiNorm{\vh}_{1,\E}^2 
\le \SE(\vh,\vh)
\quad \forall \vh \in \VhE,
\qquad
\SE(v,v)
\le \alpha^* \SemiNorm{v}_{1,\E}^2 
\quad \forall v \in H^1(\E),\; (v,1)_{0, D} =0,
\]
where~$D$ is any subset with nonzero measure
of either~$\E$ or~$\partial \E$.
\end{cor}

The stability bounds in Lemma~\ref{lemma:stability} and Corollary~\ref{corollary:stability}
can be easily extended to the case of the so-called enhanced virtual element spaces
introduced in~\cite{Ahmad-Alsaedi-Brezzi-Marini-Russo:2013}.
More precisely, define
\begin{equation} \label{c-virtual-element-space-enhanced}
\begin{split}
\VhE:=
\Big\{ \vh \in H^1(\E) 
& \ \Big| \
\Delta \vh \in \Pbb_{\p}(\E);
\quad \vh{}_{|\e} \in \Pbb_\p(\e) ; \quad
\vh{}_{|\partial\E} \in \mathcal C^0(\partial \E); \\
& \quad \int_\E (\vh-\Pinabla \vh) \mboldalpha = 0 \; \forall \vert \boldalpha \vert=\p-1,\p \Big\}.
\end{split}
\end{equation}
This space can be endowed with the same degrees of freedom of the standard space virtual element space in~\eqref{c-virtual-element-space}.
Such degrees of freedom allow us to compute the projectors
$\Pi^\nabla_{\p+2}$ and~$\Pizp$; see~\eqref{H1-projector} and~\eqref{L2-projector}.
It is immediate to check that the stabilization
\[
\SE(\uh,\vh)
:=  \hE^{-1} (\uh,\vh)_{0, \partial\E}
       + \hE^{-2} (\Pizp \uh, \Pizp \vh)_{0,\E}
    \qquad\qquad \forall \uh, \vh \in \VhE
\]
satisfies Lemma~\ref{lemma:stability} and Corollary~\ref{corollary:stability}.

\subsection{Early essential literature and later contributions} \label{subsection:literature-nodal}
The first work containing mathematical proofs on the stabilization
is~\cite{BeiraodaVeiga-Lovadina-Russo-2017}.
This and other three works
constitute the early backbone of the stability analysis of the virtual element method.
In particular, we pinpoint and rapidly describe these four works,
which were published in 2017-2018:
\begin{itemize}
    \item \cite{BeiraodaVeiga-Lovadina-Russo-2017} contains the first analysis of the stabilization in two dimensional nodal conforming virtual elements;
    an arbitrary number of edges is allowed;
    stability bounds are there derived for different stabilizations
    (including the ``dofi-dofi'' one);
    \item \cite{Brenner-Guan-Sung:2017, Brenner-Sung:2018} contain the stability analysis for two and three dimensional nodal conforming virtual elements;
    arbitrary numbers of edges and faces are allowed;
    stability bounds are derived for different stabilizations
    (including the ``dofi-dofi'' one);
    \item \cite{Chen-Huang:2018} contains the stability analysis for two dimensional nodal conforming virtual elements;
    more standard geometries are employed.
\end{itemize}
It is our opinion that
the presentation in~\cite{Chen-Huang:2018} is the simplest one
and is therefore recommended to virtual elements novices;
the presentation in~\cite{BeiraodaVeiga-Lovadina-Russo-2017} and~\cite{Brenner-Guan-Sung:2017,Brenner-Sung:2018}
is more technical and applies to more general geometries,
whence these contributions are recommended to more expert readers.
\medskip

After 2018,
other works were devoted to additional mathematical aspects of the stability in nodal conforming virtual elements.
Amongst others, we mention
nodal conforming virtual elements on curved domains~\cite{BeiraodaVeiga-Russo-Vacca:2019};
robustness with respect to anisotropic elements~\cite{Cao-Chen:2018, Cao-Chen:2019};
virtual elements stabilized by means of higher-order polynomial energy projection terms~\cite{Berrone-Borio-Marcon:2022, Berrone-Borio-Marcon:2022B, Chen-Sukumar:2023, Lamperti-Cremonesi-Perego-Russo-Lovadina:2023};
lack of robustness with respect to the degree of accuracy of the method~\cite{BeiraodaVeiga-Chernov-Mascotto-Russo:2018};
nodal serendipity virtual elements~\cite{BeiraodaVeiga-Mascotto:2023};
presence of stabilization terms in residual error estimators~\cite{BeiraodaVeiga-Manzini:2015, Cangiani-Georgoulis-Pryer-Sutton:2017}.
We shall elaborate more on the two last topics in Section~\ref{section:p-version-adaptivity} below.

The literature stability analysis for other types of elements is more limited.
The reason for this is that nodal conforming virtual element functions
have a ``second order elliptic structure''.
In other virtual elements, local spaces consist of functions solving local problems (with polynomial data)
that are in general more technical to handle.

\section{Stability in other ``elliptic-type'' virtual elements} \label{section:ncVEM}
In this section, we investigate the role of the stabilization
for other ``elliptic-type'' virtual elements
that still have an ``elliptic structure''.
In particular, we focus on ``second order'' nonconforming virtual elements in Section~\ref{subsection:nc-VE};
investigate their stability properties in Section~\ref{subsection:stability-nc};
review conforming and nonconforming arbitrarily regular virtual element spaces in Section~\ref{subsection:arbitrary-VE}.

\subsection{``Nodal'' nonconforming virtual elements} \label{subsection:nc-VE}
Following~\cite{Ayuso-Lipnikov-Manzini:2016},
given a polygonal element~$\E$,
$\EcalE$ its set of edges,
and a positive integer number~$\p$,
we define the ``nodal'' nonconforming virtual element
\begin{equation} \label{nc-virtual-element-space}
\VhE:=
\left\{ \vh \in H^1(\E) \ \middle| \
\Delta \vh \in \Pbb_{\p-2}(\E),\;
\nbfE \cdot \nabla \vh{}_{|\e} \in \Pbb_{\p-1}(\e)
\quad \forall \e \in \EcalE \right\}.
\end{equation}
We endow the space~$\VhE$ with the following set of unisolvent degrees of freedom:
given~$\vh$ in~$\VhE$,
\begin{itemize}
    \item on each edge~$\e$ of~$\E$,
    given the scaled monomial basis~$\{ \malphae \}$
    of~$\Pbb_{\p-1}(\e)$ discussed in Section~\ref{section:introduction},
    the scaled moments
    \[
    \frac{1}{\vert \e \vert} \int_\e \malphae \ \vh;
    \]
    \item given the scaled monomial basis~$\{ \mboldalpha \}$
    of~$\Pbb_{\p-2}(\E)$ as in~\eqref{monomials}, the scaled moments
    \[
    \frac{1}{\vert \E \vert} \int_\E \mboldalpha \ \vh.
    \]
\end{itemize}
We collect the above degrees of freedom in the set
$\{ \dof_j \}_{j=1}^{\NE}$,
being~$\NE$ the dimension of~$\VhE$.

The degrees of freedom of~$\VhE$ allow for the computation of several polynomial projections~\cite{Ayuso-Lipnikov-Manzini:2016}.
In what follows, we need the operators~$\Pinabla : H^1(\E) \to \Pbb_{\p}(\E)$ as in~\eqref{H1-projector},
$\Piz : L^2(\E) \to \Pbb_{\p-2}(\E)$ as in~\eqref{L2-projector},
and~$\Pize: L^2(\e) \to \Pbb_{\p-1}(\e)$
defined as
\[
(\qpmoe, v-\Pize v)_{0,\e} 
\qquad\qquad \forall \qpmoe \in \Pbb_{\p-1}(\e).
\]
Given the degrees of freedom of a function~$\vh$ in~$\VhE$,
we can compute~$\Pinabla \vh$, $\Piz \vh$,
and~$\Pize \vh$ for all edges~$\e$;
see~\cite{Ayuso-Lipnikov-Manzini:2016}.
The standard discretization of the bilinear form
$\aE(\cdot,\cdot) := (\nabla \cdot,\nabla \cdot)_{0,\E}$
in ``nodal'' nonconforming virtual elements is as in~\eqref{local-discrete-bf}.

Also in the ``nodal'' nonconforming virtual element setting,
the stabilization~$\SE(\cdot , \cdot) : \VhE \times \VhE$ is a bilinear form that is computable via the degrees of freedom,
and is such that there exist $0 <\alpha_* \le \alpha^*$
independent of~$\hE$ for which
the stability bounds in~\eqref{stability-bounds} are valid.

\subsection{Stability bounds in ``nodal'' nonconforming virtual elements} \label{subsection:stability-nc}
We introduce two stabilizations for ``nodal'' nonconforming virtual elements.
The first one is the ``dofi-dofi'' stabilization as in~\eqref{dofi-dofi-stab}.
The second one,
which we shall refer to as ``projected'' stabilization, is
\begin{equation} \label{projected-stabilization-nc}
    \SE(\uh,\vh)
    := \hE^{-1} \sum_{\e \in \EcalE} (\Pize \uh, \Pize \vh)_{0, \e}
       + \hE^{-2} (\Piz \uh, \Piz \vh)_{0,\E}
    \quad \forall \uh, \vh \in \VhE.
\end{equation}
Both stabilization are computable using the degrees of freedom of~$\VhE$.

\begin{lem} \label{lemma:stability-nc}
The bilinear forms~$\SE(\cdot ,\cdot)$ in~\eqref{dofi-dofi-stab}
and~\eqref{projected-stabilization-nc} satisfy~\eqref{stability-bounds}.
\end{lem}
\begin{proof}
We prove only the assertion for the ``projected'' stabilization~\eqref{projected-stabilization-nc};
to this aim, we follow the guidelines in~\cite[Theorem~3.2]{Mascotto-Perugia-Pichler:2018}.
The details for the ``dofi-dofi'' stabilization~\eqref{dofi-dofi-stab} are similar but slightly more involved.

The following polynomial inverse inequality holds true
\cite[Theorem~3.2]{Mascotto-Perugia-Pichler:2018}:
\begin{equation} \label{polynomial-inverse-negative}
\Norm{\qp}_{0,\partial\E}
\lesssim \hE^{-\frac12} \Norm{\qp}_{-\frac12,\partial\E}
\qquad\qquad
\forall \qp \in \mathcal C^0(\partial\E),
\quad \qp|_{|\e} \in \Pbb_{\p} (\e) \quad \forall \e \in \EcalE.
\end{equation}
\medskip

\noindent\textbf{The lower bound.}
Integrating by parts, recalling the properties of functions in ``nodal'' nonconforming virtual element spaces in~\eqref{nc-virtual-element-space},
applying the Cauchy-Schwarz inequality twice,
using the polynomial inverse inequality~\eqref{polynomial-inverse-negative},
invoking the Neumann trace inequality~\cite[Theorem~A.33]{Schwab:1998},
and recalling the virtual inverse estimate~\eqref{inverse-estimate-polynomial-Laplacian},
we can write
\[
\begin{split}
\SemiNorm{\vh}^2_{1,\E}
& = - \int_\E \Delta\vh \ \vh 
        + \sum_{\e \in \EcalE} \int_{\e} (\nbfE \cdot \nabla \vh) \vh
= - \int_\E \Delta\vh \ \Piz \vh
        + \sum_{\e \in \EcalE} \int_{\e} (\nbfE \cdot \nabla \vh) \Pize \vh \\
& \le \Norm{\Delta\vh}_{0,\E} \Norm{\Piz\vh}_{0,\E}
      + \sum_{\e \in \EcalE} \Norm{\nbfE \cdot \nabla \vh}_{0,\e} \Norm{\Pize \vh}_{0,\e} \\
& \le \Norm{\Delta\vh}_{0,\E} \Norm{\Piz\vh}_{0,\E}
      + \Norm{\nbfE \cdot \nabla \vh}_{0,\partial \E}
      \Big(\sum_{\e \in \EcalE}  \Norm{\Pize \vh}_{0,\e}^2 \Big)^{\frac12} \\
& \lesssim \Norm{\Delta\vh}_{0,\E} \Norm{\Piz\vh}_{0,\E}
      + \Norm{\nbfE \cdot \nabla \vh}_{-\frac12,\partial \E}
      \hE^{-\frac12} \Big(\sum_{\e \in \EcalE}  \Norm{\Pize \vh}_{0,\e}^2 \Big)^{\frac12} \\
& \lesssim \Norm{\Delta\vh}_{0,\E} \Norm{\Piz\vh}_{0,\E}
      + (\hE \Norm{\Delta\vh}_{0,\E} + \SemiNorm{\vh}_{1,\E})
      \hE^{-\frac12} \Big(\sum_{\e \in \EcalE}  \Norm{\Pize \vh}_{0,\e}^2 \Big)^{\frac12} \\
& \lesssim \SemiNorm{\vh}_{1,\E}
        \left( \hE^{-1} \Norm{\Piz \vh}_{0,\E}
      + \hE^{-\frac12}\Big(\sum_{\e \in \EcalE}  \Norm{\Pize \vh}_{0,\e}^2 \Big)^{\frac12} \right).
\end{split}
\]

\medskip

\noindent\textbf{The upper bound.}
Using the stability of~$\Piz$ and~$\Pize$
in the~$L^2(\E)$ and~$L^2(\e)$ (for all edges~$\e$ of~$\E$) norms,
respectively,
the trace inequality,
and the Poincar\'e inequality,
the upper bound in~\eqref{stability-bounds} follows:
\[
\begin{split}
\SE(\vh,\vh)
& := \hE^{-1} \sum_{\e \in \EcalE} \Norm{\Pize\vh}^2_{0,\e}
    + \hE^{-2} \Norm{\Piz \vh}^2_{0,\E}
\le \hE^{-1}  \sum_{\e \in \EcalE} \Norm{\vh}^2_{0,\e}
    + \hE^{-2} \Norm{\vh}^2_{0,\E}\\
& = \hE^{-1}  \Norm{\vh}^2_{0,\partial \E}
    + \hE^{-2} \Norm{\vh}^2_{0,\E}
\lesssim \SemiNorm{\vh}^2_{1,\E}.
\end{split}
\]
\end{proof}

As in the case of nodal conforming virtual elements,
Lemma~\ref{lemma:stability-nc} can be generalized
in the sense of Corollary~\ref{corollary:stability}.

Moreover, as in Section~\ref{subsection:stab-nodal},
it is possible to define an enhanced version
of the space in~\eqref{nc-virtual-element-space}:
\[
\begin{split}
\VhE:=
\Big\{ \vh \in H^1(\E) 
& \ \Big| \
\Delta \vh \in \Pbb_{\p}(\E);
\quad \nbfE \cdot \nabla \vh{}_{|\e} \in \Pbb_{\p-1}(\e)
\quad \forall \e \in \EcalE; \\
& \quad \int_\E (\vh-\Pinabla \vh) \mboldalpha = 0 \quad
    \forall \vert \boldalpha \vert=\p-1,\p \Big\}.
\end{split}
\]
We can endow this space with the same degrees of freedom
as for the space in~\eqref{nc-virtual-element-space},
which allow for the computation of higher order polynomial projectors~$\Pi^\nabla_{\p+2}$ and~$\Pi_{\p}^0$;
see~\eqref{H1-projector} and~\eqref{L2-projector}.

It can be checked that the following stabilization satisfies Lemma~\ref{lemma:stability-nc}:
\[
\SE(\uh,\vh)
:=  \hE^{-1} \sum_{\e \in \EcalE} (\Pize \uh, \Pize \vh)_{0, \e}
       + \hE^{-2} (\Pi^0_\p \uh, \Pi^0_\p \vh)_{0,\E}
    \qquad\qquad \forall \uh, \vh \in \VhE .
\]

Other theoretical results on the stabilization for ``nodal'' nonconforming virtual elements
can be found in~\cite{Bertoluzza-Manzini-Pennacchio-Prada:2022}.

\subsection{Arbitrarily regular virtual element spaces} \label{subsection:arbitrary-VE}
Arbitrarily regular virtual element spaces were one of the first generalization
of the nodal conforming virtual element method
detailed in Section~\ref{subsection:nodal-VE}
and trace back to almost the inception of the method;
see~\cite{BeiraodaVeiga-Manzini:2014, Brezzi-Marini:2013}.
Over the years, several conforming and nonconforming variants have been proposed;
see, e.g., \cite{Antonietti-Manzini-Verani:2018, Antonietti-Manzini-Verani:2020, Zhao-Chen-Zhang:2016}.

Differently from standard $\mathcal C^0$ elements,
arbitrarily regular virtual elements
are designed so as global $\mathcal C^k$ spaces, $k>0$, can be constructed.
This is accomplished by an enrichment of interface and bulk degrees of freedom.
In particular, local virtual element spaces are the set of solutions to polyharmonic problems with polynomial data.

Thence, the ``elliptic structure'' of the spaces can be still employed while deriving the stability bounds.
For this reason, the analysis is quite similar to that for standard ``nodal'' conforming and nonconforming virtual elements.
We refer to~\cite{Chen-Huang:2020} and~\cite{Chen-Huang-Wei:2022} for the stability bounds in nonconforming and conforming arbitrarily regular virtual elements.
It is quite interesting the fact that both references are rather recent (2020 and 2022, respectively).
This is probably motivated by the fact that
(\emph{i}) stability bounds are derived in any dimension;
(\emph{ii}) even though the structure of the arbitrarily regular virtual element spaces is still ``elliptic'',
the interaction of high order derivatives and the presence of multiple boundary conditions
render the analysis of the stability bounds quite involved.

We report here, for the 2D case, a stability result for $\mathcal C^1$ virtual elements.
Let~$\p\ge3$
\footnote{The case~$\p=2$ requires a slightly different but similar treatment,
and is therefore omitted.
and consider the space
\[
\VhE:=
\left\{
\vh \in H^2(\E) \ \middle| \
\Delta^2 \vh \in \Pbb_{\p-4}(\E),\;
\vh{}_{|\e} \in \Pbb_\p(\e),\;
\nbfE \cdot \nabla \vh \in \Pbb_{\p-1}(\e)
\quad \forall \e \in \EcalE
\right\}.
\]
We endow this space with the following set of degrees of freedom~\cite{BeiraodaVeiga-Manzini:2014, Huang-Yu:2021}: given~$\vh$ in~$\VhE$,
\begin{itemize}
    \item the point values of~$\vh$ at the vertices of~$\E$;
    \item the point values of~$\hE \partial_1 \vh$
    and~$\hE \partial_2 \vh$ at the vertices of~$\E$,
    where $\partial_1$ and~$\partial_2$ are the derivatives along the edges matching at such vertices;
    \item for all edges~$\e$ in~$\EcalE$,
    given~$\{ \malphae \}$ the scaled and shifted monomial basis of~$\Pbb_{\p-4}(\e)$,
    the moments
    \[
    \frac{1}{\he} \int_{\e} \malphae \ \vh;
    \]
    \item for all edges~$\e$ in~$\EcalE$,
    given~$\{ \malphae \}$ the scaled and shifted monomial basis of~$\Pbb_{\p-3}(\e)$,
    the moments
    \[
    \int_{\e} \malphae \ \nbfE \cdot \nabla \vh;
    \]
    \item given~$\{ \mboldalpha \}$ the scaled and shifted monomial basis of~$\Pbb_{\p-4}(\e)$,
    the moments
    \[
    \frac{1}{\vert \E \vert} \int_{\E} \mboldalpha \ \vh.
    \]
\end{itemize}
Based on such degrees of freedom, the following stabilization is computable:
\begin{equation} \label{stabilization-biharmonic}
\SE(\uh,\vh)
:= \hE^{-4} (\Pi^0_{\p-4}\uh, \Pi^0_{\p-4}\vh)_{0,\E}
   + \hE^{-3} (\uh,\vh)_{0,\partial \E}
   + \hE^{-1} (\nbfE\cdot\nabla\uh, \nbfE\cdot\nabla\vh)_{0,\partial \E} .
\end{equation}
Proceeding as in Section~\ref{subsection:stab-nodal}
and~\cite{Chen-Huang-Wei:2022} yields the following result.
\begin{lem}
The stabilization in~\eqref{stabilization-biharmonic}
satisfies the two following bounds:
there exist~$0< \alpha_* \le \alpha^*$ such that
\[
\begin{split}
& \alpha_* \SemiNorm{\vh}_{2,\E}^2 
\le \SE(\vh,\vh)
\quad \forall \vh \in \VhE,\\
& \SE(v,v)
\le \alpha^* \SemiNorm{v}_{2,\E}^2 
\qquad \forall v \in H^1(\E)
\quad \text{ such that } \quad
(v,1)_{0, D} =0,
\; (\nabla v, (1,1))_{0,D}=0,
\end{split}
\]
where~$D$ is any subset with nonzero measure
of either~$\E$ or~$\partial \E$.
\end{lem}
}

\section{Stability in ``nonelliptic'' virtual elements} \label{section:non-elliptic}
In this section, we review the literature on the stabilization
for some ``nonelliptic'' virtual elements:
face (Section~\ref{subsection:face-VE});
edge (Section~\ref{subsection:edge-VE});
Stokes-like (Section~\ref{subsection:Stokes-VE});
immersed (Section~\ref{subsection:immersed-VE}) virtual elements.

\subsection{Face virtual elements} \label{subsection:face-VE}
Face virtual elements were designed in~\cite{Brezzi-Falk-Marini:2014, BeiraodaVeiga-Brezzi-Marini-Russo:2016-mixed, BeiraodaVeiga-Brezzi-Marini-Russo:2016-divcurl}
and generalize the Raviart-Thomas elements to polytopic meshes.
Both in two and three dimensions,
local face spaces consist of functions solving local ``div-rot'' problems and with polynomial normal components on faces.
The degrees of freedom consist of bubble-like moments in the element
and moments of normal components over faces.

The first stability analysis (on 2D curved elements) can be found in~\cite{Dassi-Fumagalli-Losapio-Scialo-Scotti-Vacca:2021}.
Lowest order two and three dimensional face spaces were analyzed in~\cite{BeiraodaVeiga-Mascotto:2022},
general order two and three dimensional (standard and serendipity) face spaces in~\cite{BeiraodaVeiga-Mascotto-Meng:2022},
and three dimensional face spaces on curved polyhedra in~\cite{Dassi-Fumagalli-Scotti-Vacca:2022}.
Related results are discussed in~\cite{Zhao-Zhang:2021}.

Compared to the elliptic case, the stability analysis
is here complicated by the ``div-rot'' structure of the spaces.
Notably, the analysis needs the design of certain Helmholtz-like decompositions,
as well as employing more sophisticated results from the theory of vector potential and div-rot systems;
see, e.g., \cite{Costabel:1990, Amrouche-Bernardi-Dauge-Girault:1998}.

\subsection{Edge virtual elements} \label{subsection:edge-VE}
Edge virtual elements were first designed in~\cite{BeiraodaVeiga-Brezzi-Marini-Russo:2016-divcurl}
and generalize the N\'ed\'elec element to polytopic elements.
In three dimensions, local edge spaces consist of functions solving local ``div-curlcurl'' problems inside the element;
local ``div-rot'' problems on faces;
polynomial tangential traces over edges.
The degrees of freedom consist of bubble-like moments in the element
and on faces, and moments of tangential components over edges.

The first stability analysis (lowest order, in two and three dimensions)
can be found in~\cite{BeiraodaVeiga-Mascotto:2022};
the general order (standard and serendipity) case was later tackled in~\cite{BeiraodaVeiga-Mascotto-Meng:2022}.

Compared to the the case of face elements,
the analysis of edge elements in three dimensions
is further complicated by the fact
that virtual element functions are solutions to different types of problems inside the element (``div-curlcurl'' problems)
and on faces (``div-rot'' problems),
which require a different treatment to establish
suitable stability estimates.

\subsection{Stokes-like virtual elements} \label{subsection:Stokes-VE}
Lowest order Stokes-like virtual elements appeared in~\cite{Antonietti-BeiraodaVeiga-Mora-Verani:2014}
and were later generalized to the general order case in~\cite{BeiraodaVeiga-Lovadina-Vacca:2017}.
Local spaces consist of functions solving local Stokes-like
with polynomial data;
in particular, the divergence of Stokes-like virtual element functions is polynomial.
This allows for the design of methods that are automatically divergence free.
In two dimensions, the degrees of freedom consist of bubble-like moments in the element and point values on the boundary;
in three dimensions, of element and face moments,
and point values on the edges.

Stokes-like virtual elements are very successful
and were the topic of a plethora of articles.
However, to the best of our knowledge,
the first and only paper dealing with the stability analysis
of Stokes-like virtual element spaces is~\cite{BeiraodaVeiga-Mascotto-Meng:2023}.
The stability bounds are essentially derived based on the stable inf-sup structure of the local Stokes problems,
and polynomial and (novel) virtual element inverse estimates.

\subsection{Immersed-like virtual elements} \label{subsection:immersed-VE}
More recently, immersed-like virtual elements have been introduced in two and three dimensions
for elliptic and Maxwell problems;
see~\cite{Cao-Chen-Guo:2022, Cao-Chen-Guo-Lin:2022}.
Such spaces are particularly effective when handling problems
with nonsmooth coefficients.
The irregular behaviour of the exact solution
is captured by local immersed-like virtual element functions in a Trefftz fashion,
as they are solutions to local problems involving the same nonsmooth coefficients as in the continuous problem.

\section{Dependence on the degree of accuracy and error estimators} \label{section:p-version-adaptivity}
Two situations where the presence of the stabilization
might be troublesome are
the $\p$- and $\h\p$-versions of the method,
and when proving the reliability and efficiency of error estimators.
We elaborate on these two topics
in Sections~\ref{subsection:p-version} and~\ref{subsection:residual},
respectively.

\subsection{Stability bounds depending on the degree of accuracy} \label{subsection:p-version}
The $\p$- and $\h\p$-versions of a Galerkin method
aim at achieving convergence
by increasing the dimension of the local approximation spaces (while keeping the mesh fixed),
and by mesh refinement and $\p$-version at once, respectively.

Stability bounds with explicit dependence on the degree of accuracy~$\p$ were investigated in several works.
Amongst them, we recall the following contributions:
$\p$-explicit stability bounds were first derived in~\cite{BeiraodaVeiga-Chernov-Mascotto-Russo:2018}
for the ``projected'' stabilization in~\eqref{projected-stabilization};
this result was extended~\cite{Antonietti-Mascotto-Verani:2018} to the ``dofi-dofi'' stabilization in~\eqref{dofi-dofi-stab};
the $\p$-version nonconforming virtual elements
was addressed in~\cite{Mascotto-Perugia-Pichler:2018}.

In all cases, at least one of the constants~$\alpha_*$ and~$\alpha^*$ in~\eqref{stability-bounds} depend on~$\p$.
This is not surprising as several inverse estimates are employed to derive the lower bound in~\eqref{stability-bounds};
see the lower bounds in Lemmas~\ref{lemma:stability} and~\ref{lemma:stability-nc}.
However, the suboptimality with respect to the degree of accuracy
is rather mild in practice;
see~\cite[Section~4.1]{BeiraodaVeiga-Chernov-Mascotto-Russo:2018}.

\subsection{The stability in residual error estimators} \label{subsection:residual}
In finite elements,
an error estimator~$\eta$
given by the combination of local error estimators~$\etaE$
is a quantity that can be computed by means of the discrete solution~$\uh$ to the method
and has to be designed so as to be comparable
to the error of the method.

For the Poisson problem,
standard reliability and efficiency estimates read
\begin{equation} \label{reliability-efficiency-FEM}
\SemiNorm{u-\uh}_{1,\Omega}
\lesssim 
\eta(\uh) + \HOT_{\Omega}(f),
\qquad\qquad
\etaE(\uh) \lesssim
\SemiNorm{u-\uh}_{1,\omegaE}
+ \HOT_{\omegaE}(f),
\end{equation}
where~$\omegaE$ is the patch of elements around~$\E$
and~$\HOT_D(f)$ is a high-order oscillation term on a domain~$D$ involving the right-hand side~$f$.

Residual error estimators in virtual elements~\cite{BeiraodaVeiga-Manzini:2015, Cangiani-Georgoulis-Pryer-Sutton:2017}
satisfy analogous bounds that differ from~\eqref{reliability-efficiency-FEM}
in two aspects:
(\emph{i}) the discrete solution~$\uh$ is not available in closed form, so it must be replaced by a polynomial projection;
(\emph{ii}) the stabilization appears in the error estimator
and error estimates.

So, for a computable virtual element error estimator~$\etatilde$
given by the combination of local virtual element error estimators~$\etatildeE$,
the virtual element counterpart of~\eqref{reliability-efficiency-FEM} has the following form:
\[
\begin{split}
& \SemiNorm{u-\uh}_{1,\Omega}
\lesssim 
\etatilde(\Pinabla \uh)
+ \sum_{\E \in \taun} \SE((I-\Pinabla)\uh, (I-\Pinabla)\uh)^{\frac12}
+ \HOT_\Omega(f),\\
& \etatildeE(\Pinabla \uh) \lesssim
\SemiNorm{u-\uh}_{1,\omegaE}
+ \!\!\!\!\sum_{\E \in \taun,\ \E \subset\omegaE}\!\!\!\!
            \SE((I-\Pinabla)\uh, (I-\Pinabla)\uh)^{\frac12}
+ \HOT_{\omegaE}(f) .
\end{split}
\]
The presence of the extra stabilization term in the bounds above
might be inauspicious.
On the one hand, such term
is not robust with respect to the degree of accuracy
as discussed in Section~\ref{subsection:p-version}.
On the other hand,
adaptive mesh refinements may lead to polygonal meshes
with several hanging nodes and possibly very small facets;
in that case, one should be extremely careful
in designing stabilization that are robust with respect to badly-shaped elements.

A partial breakthrough facing these issues
is contained in~\cite{BeiraodaVeiga-Canuto-Nochetto-Vacca-Verani:2022},
where the stabilization term is removed in the reliability and efficiency estimates.
However, this has been proved under restrictive assumptions:
meshes must consist of elements with triangular shape;
the 3D version is not covered;
lowest order elements are employed.
So, general reliability and efficiency estimates that are robust
with respect to the stabilization are currently not available,
albeit the subject of current study.

\section{The stability bounds may imply interpolation estimates} \label{section:interpolation}
Interpolation estimates by means of functions in the virtual element space
are an important ingredient in the error analysis of the method.
In this section, we first report standard ways of proving interpolation estimates
in conforming and nonconforming virtual elements;
see Sections~\ref{subsection:interpolation-standard-conforming} 
and~\ref{subsection:interpolation-standard-nonconforming}, respectively.
Next, in Section~\ref{subsection:interpolation-new},
we show a more recent strategy
to derive interpolation estimates based on the stability bounds
covering the conforming and nonconforming cases at once.

\subsection{Standard proof for interpolation estimates:
conforming elements} \label{subsection:interpolation-standard-conforming}
The first proof of interpolation estimates
for nodal conforming virtual elements traces back to 2015;
see~\cite{Mora-Rivera-Rodriguez:2015}.
Several variants have been proposed ever since,
most of them extending the ideas therein.

Here, we only provide details on the two dimensional nodal conforming virtual elements;
see~\cite[Proposition~4.2]{Mora-Rivera-Rodriguez:2015}.
Given an element~$\E$ split into a shape-regular triangulation~$\tauE$,
$v$ in~$H^1(\E)$,
$\qp$ in~$\Pbb_\p(\E)$,
and~$\vC$ the Cl\'ement quasi-interpolant of~$v$ over~$\tauE$,
we define~$\vI$ over~$\E$ as the weak solution to
\[
\begin{cases}
    -\Delta \vI = - \Delta \qp & \text{in } \E\\
    \vI         = \vC           & \text{on } \partial \E.
\end{cases}
\]
A minimum energy argument entails
\[
\SemiNorm{v-\vI}_{1,\E}
\le 2 \SemiNorm{v-\qp}_{1,\E}
          + \SemiNorm{v-\vC}_{1,\E},
\]
whence interpolation estimates follow from 
polynomial quasi-interpolation
and approximation estimates.

Thus, deriving interpolation estimates for conforming virtual elements
strongly hinges upon using the structure of the local virtual element spaces,
i.e., the stability of the continuous formulation.
When considering other types of virtual elements,
this may result in interpolation estimates
that are much harder to derive.
Prototypical examples of this fact are interpolation estimates in Stokes-like~\cite{BeiraodaVeiga-Lovadina-Vacca:2017};
Hellinger-Reissner-like~\cite{Artioli-DeMiranda-Lovadina-Patruno:2017, Dassi-Lovadina-Visinoni:2020};
face and edge~\cite{BeiraodaVeiga-Mascotto:2022, BeiraodaVeiga-Mascotto-Meng:2022} virtual elements.

\subsection{Standard proof for interpolation estimates:
nonconforming elements} \label{subsection:interpolation-standard-nonconforming}
The first proof of interpolation estimates
for ``nodal'' nonconforming virtual elements traces back to 2018;
see~\cite{Mascotto-Perugia-Pichler:2018}.
Several variants have been proposed ever since,
most of them extending the ideas therein.

Here, we only provide details on the two dimensional ``nodal'' nonconforming virtual elements;
see~\cite[Proposition~3.1]{Mascotto-Perugia-Pichler:2018}
and the later work~\cite[Corollary~4.1]{Huang-Yu:2021}.

Given~$\E$ in~$\taun$ and~$v$ in~$H^1(\E)$,
we define~$\vI$ in the ``nodal'' nonconforming space~$\VhE$
defined in~\eqref{nc-virtual-element-space}
as the degrees of freedom interpolant of~$v$.
Notably, $\vI$ is uniquely identified by
\[
\int_\E (v-\vI) \qpmtE = 0 \quad \forall \qpmtE \in \Pbb_{\p-2}(\E),
\qquad
\int_\e (v-\vI) \qpmoe = 0 \quad \forall \qpmoe \in \Pbb_{\p-1}(\e)
\quad \forall \e \in \EcalE.
\]
For any~$\qpE$ in~$\Pbb_\p(\E)$,
recalling that~$\Delta \vI$ belongs to~$\Pbb_{\p-2}(\E)$
and~$\nbfE \cdot \nabla \vI{}_{|\e}$ belongs to~$\Pbb_{\p-1}(\e)$
for all the edges~$\e$ of~$\E$,
we readily deduce
\[
\begin{split}
\SemiNorm{v-\vI}_{1,\E}
& = - \int_{\E} \Delta (v-\vI) \ (v-\vI) 
    + \sum_{\e \in \EcalE} \int_\e \nbfE \cdot \nabla (v-\vI)\ (v-\vI)\\
& = - \int_{\E} \Delta (v-\qpE) \ (v-\vI) 
    + \sum_{\e \in \EcalE} \int_\e \nbfE \cdot \nabla (v-\qpE)\ (v-\vI)\\
& = \int_\E \nabla(v-\qpE) \cdot \nabla(v-\vh)
  \le \SemiNorm{v-\qpE}_{1,\E}   \SemiNorm{v-\vI}_{1,\E},
\end{split}
\]
whence we deduce
\[
\inf_{\vh \in \VhE} \SemiNorm{v-\vh}_{1,\E}
\le \SemiNorm{v-\vI}_{1,\E}
\le \inf_{\qp \in \Pbb_\p(\E)} \SemiNorm{v-\qpE}_{1,\E} .
\]
Extending this approach to other types of virtual elements
may be not immediate.

\subsection{A more recent proof for interpolation estimates} \label{subsection:interpolation-new}
There is a recent alternative approach originally employed in~\cite{BeiraodaVeiga-Mascotto-Meng:2023}
to derive interpolation estimates,
which applies to conforming and nonconforming virtual elements at once,
and is based on having stability bounds at hand.

Again, to simplify the presentation we stick to the two dimensional nodal
conforming virtual element case
and consider the stabilization in~\eqref{dofi-dofi-stab},
even though the following steps might be generalized
to several other elements and dimensions,
as well as to other stabilizations.
Let~$\E$ be a given element, $v$ be in~$H^s(\E)$, $s>1$,
and~$\qp$ be any function in~$\Pbb_\p(\E)$
with the same average of~$v$ over~$\E$.
We denote the degrees of freedom interpolant of any sufficiently smooth function
by adding a subscript~$\cdot_I$ to such function.
Observe that~$\qp = (\qp)_I$.
Recalling Corollary~\ref{corollary:stability}
and using that~$\SE(\vI,\vI) = \SE(v,v)$
by definition~\eqref{dofi-dofi-stab}, we readily get
\[
\begin{split}
\SemiNorm{v-\vI}_{1,\E}
& \le \SemiNorm{v-\qp}_{1,\E} + \SemiNorm{\vI-\qp}_{1,\E}
  \le \SemiNorm{v-\qp}_{1,\E} 
  + \alpha_*^{-1} \SE(\vI-\qp,\vI-\qp)^{\frac12}\\
& = \SemiNorm{v-\qp}_{1,\E} 
  + \alpha_*^{-1} \SE((v-\qp)_I,(v-\qp)_I)^{\frac12}
  = \SemiNorm{v-\qp}_{1,\E} 
  + \alpha_*^{-1} \SE(v-\qp, v-\qp)^{\frac12}\\
& \le \left(1 + \frac{\alpha^*}{\alpha_*} \right)
        \SemiNorm{v-\qp}_{1,\E}.
\end{split}
\]
Thus, upon having at disposal stability bounds of the form~\eqref{stability-bounds},
we proved that the best interpolation error is bounded
by the best polynomial error up to the stability constants.
The above estimates readily extend to nonconforming nodal virtual elements.

\section{Conclusions and outlook} \label{section:conclusions}
The first paper on the virtual element method
was published ten years ago.
The analysis of the stabilization of the method
is more recent and is still a current and lively area of research.
Several works on the stability in nodal conforming virtual elements
have been published; fewer on other types of virtual elements.
In this contribution, we reviewed the literature
about the role and theoretical analysis
of the stabilization in the virtual element method.
We presented paradigmatic proofs for ``nodal'' conforming and nonconforming virtual elements;
we only mentioned the main results for other types of elements.
We further underlined that the stability bounds
imply interpolation estimates.



{\footnotesize
\bibliography{bibliogr.bib}}
\bibliographystyle{plain}

\end{document}